\documentclass[oneside,english]{amsart}
\usepackage[T1]{fontenc}
\usepackage[latin9]{inputenc}
\usepackage{amsthm}
\usepackage{amstext}
\usepackage{amssymb}
\usepackage{esint}
\usepackage{float}
\restylefloat{table}

\makeatletter

\providecommand{\tabularnewline}{\\}

\numberwithin{equation}{section}
\numberwithin{figure}{section}
\theoremstyle{plain}
\newtheorem{thm}{\protect\theoremname}[section]
  \theoremstyle{definition}
  \newtheorem{defn}[thm]{\protect\definitionname}
  \theoremstyle{plain}
  \newtheorem{prop}[thm]{\protect\propositionname}
  \theoremstyle{remark}
  \newtheorem{rem}[thm]{\protect\remarkname}
  \theoremstyle{plain}
  
  \theoremstyle{plain}
  \newtheorem{cor}[thm]{\protect\corollaryname}
  \theoremstyle{plain}
  \newtheorem{lem}[thm]{\protect\lemmaname}

\makeatletter
\@addtoreset{thm}{section}
\makeatother

\makeatother

\usepackage{babel}
  \providecommand{\corollaryname}{Corollary}
  \providecommand{\definitionname}{Definition}
  \providecommand{\factname}{Fact}
  \providecommand{\lemmaname}{Lemma}
  \providecommand{\propositionname}{Proposition}
  \providecommand{\remarkname}{Remark}
\providecommand{\theoremname}{Theorem}

\begin{document}

\title{Riemannian Holonomy Groups of Statistical Manifolds}

\thanks{This subject is supported by the National
Natural Science Foundations of China (No. 61179031, No. 10932002.)
}

\author[D. Li]{Didong Li}
\address{School of Mathematics and Statistics, Beijing Institute
              of Technology, Beijing 100081, P. R. China}
\email{lididong@gmail.com}

\author[H. Sun]{Huafei Sun}
\address{School of Mathematics and Statistics, Beijing Institute
              of Technology, Beijing 100081, P. R. China}
\email{huafeisun@bit.edu.cn}

\thanks{The second author is the corresponding author}

\author[C. Tao]{Chen Tao}
\address{School of Mathematics and Statistics, Beijing Institute
              of Technology, Beijing 100081, P. R. China}
\email{matheart@gmail.com}

\author[L. Jiu]{Lin JIu}
\address{Department of Mathematics, Tulane University, New Orleans, LA 70118, U.S.A. }
\email{ljiu@tulane.edu}

\begin{abstract}
Normal distribution manifolds play essential roles in the theory of
information geometry, so do holonomy groups in classification of Riemannian manifolds. After some necessary preliminaries on information
geometry and holonomy groups, it is presented that the corresponding
Riemannian holonomy group of the $d$-dimensional normal distribution is
$SO\left(\frac{d\left(d+3\right)}{2}\right)$, $d=1,2,3$.
As a generalization on exponential family, a list of holonomy groups
follows.
\end{abstract}



\maketitle

\section{Introduction}

Statistical manifolds, which consist of probability distribution functions,
are the main objects in information geometry. In order to describe
their geometric structures, a series of concepts such as $\alpha$-connections,
dual connections, and particularly Fisher metric\cite{key-1}, which
makes statistical manifold Riemannian manifold, are introduced and
studied. As a special example, normal distribution manifolds, defined
in Definition5.1, are of great importance. When S. Amari initiated
the theory of information geometry\cite{key-2,key-3}, he found that
the sectional curvature of the monistic normal distribution manifold
is $-\frac{1}{2}$, which is recalled in the proof of Lemma 5.7, implying
its isometry to a hyperbolic space. Rather than an amazing result,
it is also the trigger for Amari to develop information geometry.
Some basic defintions and results on information geometry are presented
in Section 2.

Around 1926, \'{E}. Carten introduced holonomy groups in order to study
and classify symmetric spaces. Indeed, he has classified irreducible
symmetric spaces by considering holonomy groups. As part of the generalization
of parallel tranpotations, holonomy could be defined on any vector
boudle with connections\cite{key-6}. Hence, as S. S. Chern believed, it
plays an important role in the theory of connections. However there
are seldom brilliant results except for Ambrose-Singer holonomy theorem
\cite{key-8,key-9}. When coming to Riemannian holonomy groups, the
classification for irreducible cases was solved in $1955$ by M. Berger\cite{key-10}
and J. Simons\cite{key-11}. After introducing defintions and propsitions
on holonomy groups in Section 3, several useful results on classification are presented
in terms of theorems and corollaries in Section 4 and also in terms
of two tables in Appendix. Although, holonomy group fails to classify
non-isometric manifolds due to its small number of classes, it is
still essential and is applied to many fields including string theory.

In this paper, we concentrate on the Riemannian holonomy groups of
statistical manifolds, especially the normal distribution manifolds.
After calculating the holonomy groups of normal distribution manifolds
in Theorem 5.3, Theorem 6.3 follows as a generalization on exponential
family.

\section{Information Geometry on Statistical Manifolds}

We call
\[
S:=\left\{ p\left(x;\theta\right)|\theta\in\Theta\right\}
\]
a statistical manifold if $x$ is a random variable in sample space
$X$ and $p\left(x;\theta\right)$ is the probability density function,
which satisfies certain regular conditions. Here, $\theta=\left(\theta_{1},\theta_{2},\dots,\theta_{n}\right)\in\Theta$
is an $n$-dimensional vector in some open subset $\Theta\subset\mathbb{R}^{n}$,
and $\theta$ can be viewed as the coordinates on manifold $S$.
\begin{defn}
An $n$-dimensional parametric statistical model $\Theta=\left\{ p_{\theta}|\theta\in\Theta\right\} $
is called an exponential family or of exponential type, if the probability
density function can be expressed in terms of functions $C,F_{1},\dots,F_{n}$
and a convex function $\phi$ on $\Theta$ of the form (\cite{key-38})
\[
p(x;\theta)=\exp\left\{ C(x)+\Sigma_{i}\theta_{i}F_{i}(x)-\phi(\theta)\right\}.
\]
In addition, we call $\left\{ \theta_{i}\right\} $ the natural parameters
and $\phi$ the potential function.
\end{defn}

\begin{defn}
The Riemannian metric on statistical manifolds is defined by the Fisher
information matrix\cite{key-1}:
\[
g_{ij}(\theta):=E[(\partial_{i}l)(\partial_{j}l)]=\int(\partial_{i}l)(\partial_{j}l)p(x;\theta)dx,\ \ i,j=1,2,\ldots,n,
\]
where $E$ denotes the expectation, $\partial_{i}:=\frac{\partial}{\partial\theta_{i}}$,
and $l:=l(x;\theta)=\log p(x;\theta)$.
\end{defn}
\begin{prop}
Suppose $S$ is an exponential family with coordinates $\theta_i$ and potential function $\phi$. Then the Fisher metric is given by
$$g_{ij}=\partial_i\partial_j\phi.$$
\end{prop}
\begin{defn}
A family of connections $\nabla^{(\alpha)}$ defined by Amari as follows
\[
<\nabla_{A}^{(\alpha)}B,C>:=E[(ABl)(Cl)]+\frac{1-\alpha}{2}E[(Al)(Bl)(Cl)]
\]
are called $\alpha$-connections, where $A,B,C\in\mathfrak{X}(S)$,
$ABl=A(Bl)$, and $\alpha\in\mathbb{R}$ is the parameter.\end{defn}
\begin{prop}
The connection $\nabla^{(\alpha)}$ is torsion free for all $\alpha$.
The only Riemannian connection, with respect to Fisher metric, is
$\nabla=\nabla^{(0)}$.\end{prop}

\begin{rem}
In this paper, we focus on the Riemannian case.
\end{rem}
\begin{thm}
If the Riemannian connection coefficients and $\alpha$-connection
coefficients are denoted by $\Gamma_{ijk}$ and $\Gamma_{ijk}^{(\alpha)}$,
respectively, then
\[
\Gamma_{ijk}^{(\alpha)}=\Gamma_{ijk}-\frac{\alpha}{2}T_{ijk},
\]
where $T_{ijk}:=E[(\partial_{i}l)(\partial_{j}l)(\partial_{k}l)]$.
Note that $\Gamma_{ijk}^{(0)}=\Gamma_{ijk}$, which coincides with
Proposition 2.4.\end{thm}
\begin{prop}
Suppose $S$ is an exponential family with coordinates $\theta_i$ and potential function $\phi$. Then the Riemannian connection coefficients are given by
$$\Gamma_{ijk}=\partial_i\partial_j\partial_k\phi.$$
\end{prop}
\begin{defn}
The Riemannian curvature tensor of $\alpha$-connections is defined
by
\[
R_{ijkl}^{(\alpha)}=(\partial_{j}\Gamma_{ik}^{(\alpha)s}-\partial_{i}\Gamma_{jk}^{(\alpha)s})+(\Gamma_{jtl}^{(\alpha)}\Gamma_{ik}^{(\alpha)t}-\Gamma_{itl}^{(\alpha)}\Gamma_{jk}^{(\alpha)t}),
\]
where $\Gamma_{jk}^{(\alpha)s}=\Gamma_{jki}^{(\alpha)}g^{is}$ and $\left(g^{is}\right)$
is the inverse matrix of the metric matrix $\left(g_{mn}\right)$.
Einstein notation is also used here.
\end{defn}


\section{Holonomy Groups}

\subsection{Holonomy of a Connection in a Vector Bundle}
\begin{defn}
Let $E$ be a rank $r$ vector bundle over a smooth manifold $M$
and $\nabla$ be a connection on $E$. Given a piecewise smooth loop
$\gamma:[0,1]\rightarrow M$ based at $x\in M$, we let $P_{\gamma}$
denote the parallel transportation of $\nabla$ and $E_{x}$ denote
the fiber over $x$. Then, $P_{\gamma}:E_{x}\rightarrow E_{x}$ is
an invertible linear transformation, hence an element of $GL(E_{x})\cong GL(r,\mathbb{R})$.
The holonomy group of $\nabla$ based at $x$ is defined by
\[
H_{x}(\nabla):=\{P_{\gamma}\in GL(E_{x})|\gamma\text{ is a loop based at }x\}.
\]
The restricted holonomy group based at $x$ is its subgroup defined
by
\[
H_{x}^{0}(\nabla):=\{P_{\gamma}\in GL(E_{x})|\gamma\text{ is a contratible loop based at }x\}.
\]
\end{defn}
\begin{prop}
If $M$ is connected(hence path connected), then the holonomy groups
on different base points are conjugate of one another in $GL(r,\mathbb{R})$.
In concrete, if $x,y\in M$, and $\gamma$ is a path from $x$ to
$y$, then
\[
H_{y}(\nabla)=P_{\gamma}H_{x}(\nabla)P_{\gamma}^{-1}.
\]

\end{prop}
As a result, we shall always omit the base point and denote the group
by $H(\nabla)$. While only considering one connection $\nabla$,
we could further reduce the notation for the group by $H$. Now, here are
several properties for the holonomy groups.
\begin{prop}
Let $E$ be a rank $r$ vector bundle over a connected manifold $M$,
and $\nabla$ be a connection on $E$, then

(1) $H^{0}$ is a conneted, Lie-subgroup of $GL(r,\mathbb{R})$;

(2) $H^{0}$ is the identity component of $H$, hence the determinant
of every matrix element is positive;

(3) if, in addition, $M$ is simply connected, then $H^{0}=H$;

(4) $\nabla$ is flat iff $H^{0}(\nabla)=0$.
\end{prop}

\subsection{Riemannian Holonomy}
\begin{defn}
The Riemannian holonomy group of a Riemannian manifold $(M,g)$ is
just the holonomy group of the Levi-Civita connection $\nabla$ on
the tangent bundle $TM$.
\end{defn}
In another word, the Riemannian holonomy is a special case.
\begin{prop}
Let $M$ be an $n$-dimensional Riemannian manifold and $H$ denote
its Riemannian holonomy group, then

(1) $H$ is a (compact) (closed) Lie-subgroup of $O(n)$($n$-dimensional orthogonal group);

(2) if $M$ is orientable, then $H$ is a subgroup of the special orthogonal
group $SO(n)$.
\end{prop}

\section{Classification of Riemannian Holonomy Groups}
\begin{thm}
Every locally symmetric Riemannian manifold is locally isometric to
a symmetric space.
\end{thm}
Hence, we only need to consider symmetric spaces. We begin with the
de Rham decomposition theorem.
\begin{thm}
(de Rham) Suppose that $M$ is a complete, simply connected Riemannian
manifold, then it is isometric to $\mathbb{R}^{k}\times M^{1}\times\dots\times M^{m}$,
where $k\geqslant0$ and each $M^{i}$ is an irreducible, complete,
and simply connected Riemannian manifold. Moreover, the dimension
$k$ and manifolds $M^{1},\dots,M^{m}$ are uniquely(up to the
order) determined by $M$.\end{thm}
\begin{cor}
Let $M=\mathbb{R}^{k}\times M^{1}\times\dots\times M^{m}$ be the
de Rham decomposition, $H$ be the holonomy group of $M$, and $H_{i}$
be the holonomy of $M^{i}$, then $H\cong H_{1}\times\dots\times H_{m}$.
\end{cor}
By de Rham decomposition, simply connected irreducible symmetric spaces
are essential. The holonomy of a symmetric space can be derived by
the holonomy of its factors. Therefore, we only need to find all simply
connected irreducible symmetric spaces and their holonomy groups.

In fact, all simply connected irreducible symmetric spaces $M$ are
of the form $M\cong G/K$, where $G$ is the group of isometric transformations
on $M$ and $K$ is its isotropy subgroup. There are three types of
such spaces( where $\kappa$ denotes the curvature of $M$):\\
(1) Euclidean type: $\kappa=0$ and $M$ is isometric to a Euclidean
space;\\
(2) compact type: $\kappa\geqslant0$(not identically 0);\\
(3) non-compact type: $\kappa\leq0$(not identically 0).\\
In all cases there are two classes:\\
Class A: $G$ is a real simple Lie group;\\
Class B: $G$ is either the product of a compact simple Lie group
with itself (compact type), or a complexification of such a Lie group
(non-compact type).

All these types are completly classified by \'{E}. Cartan. Please also
see \cite{key-26} for details. We only give one of the four tables
of symmetric spaces.
\begin{thm}
(\'{E}. Cartan) The seven infinite series and twelve exceptional Riemannian
symmetric spaces in Table 1(in Appendix) give all Riemannian symmetric spaces of
class A and non-compact type(called type \mbox{II}).
\end{thm}
Based on Theorems 4.1-4.4, the holonomy groups of a locally symmetric
Riemannian manifold are completely classified. The remaining problem
is to classify all non-locally symmetric Riemannian manifolds with
an irreducible holonomy group. This problem is solved by M. Berger(\cite{key-10})
in 1955 and J. Simons(\cite{key-11}) in 1962 in terms of the following
theorem.
\begin{thm}
(M. Berger) The complete classification of possible holonomy groups
for simply connected Riemannian manifolds which are irreducible and
nonsymmetric is in Table 2 (in Appendix).
\end{thm}
From Berger's list, several direct corollaries follow.
\begin{cor}
Let $n=dim(M)$ and $H$ be the holonomy group of $M$. Then

(1) if $n$ is odd and $n\neq7$, $H=SO(n)$;

(2) if $n=7$, $H=SO(7)$ or $G_{2}$;

(3) if $M$ is not an Einstein manifold and $n$ is even, $H=SO(n)$
or $U(\frac{n}{2})$.
\end{cor}

Actually, all the cases on Berger's list occur, which means for every
group $H$ on the list, there exists a manifold that admits $H$ as
its holonomy group.
\begin{rem}
The content about holonomy groups and symmetric spaces mentioned in
Section 3 and 4 is discussed in \cite{key-4}-\cite{key-37}.
\end{rem}

\section{Holonomy Groups of Normal Distribution Manifolds}
\begin{defn}
Let $PD(d,\mathbb{R})$ be the set of all real $d$-ordered positive definite
symmetric matrices. The $d$-dimensional normal distribution manifold
is defined by
\[
N^{d}:=\left\{ p(x,\mu,\Sigma)=\frac{\exp\{-\frac{1}{2}(x-\mu)^{T}\Sigma^{-1}(x-\mu)\}}{(2\pi)^{\frac{d}{2}}(\det\Sigma)^{\frac{1}{2}}}\middle|\mu\in\mathbb{R}^{d},\Sigma\in PD(d,\mathbb{R})\right\} .
\]
\end{defn}
\begin{rem}
$\,$

(1)Here the dimension $d$ is the dimension of normal distributions.
As a manifold, the dimension is not hard to compute by
\[
\dim\left(N^{d}\right)=\dim\left(\mathbb{R}^{d}\times PD\left(d,\mathbb{R}\right)\right)=d+\frac{d\left(d+1\right)}{2}=\frac{d\left(d+3\right)}{2}.
\]

(2)Also, S. Amari proved that lower dimensional normal distribution
manifold can be embedded into higher dimensional ones, i.e. if $d_{1}<d_{2}$,
we could have $N^{d_{1}}\subset N^{d_{2}}$. This is bacause lower
distributions could be treated as higher distributions with restrictions.
\end{rem}
Obviously, $N^{d}$ is also an exponential family as in Definition
2.1. Our main result is the following theorem.
\begin{thm}
Let $N^{d}$ be the d-dimensional normal distribution manifold, $g$
be the Fisher metric and $\nabla=\nabla^{(0)}$ be the corresponding
Levi-Civita connection. Suppose $H_{d}$ is the Riemannian holonomy
group and $H_{d}^{0}$ is the restricted Riemannian holonomy group,
then
\[
H_{d}=H_{d}^{0}=SO(\frac{d(d+3)}{2}),\ d=1,2,3.
\]
\end{thm}
\begin{rem}
Together with Remark 5.2, we recognize the result as the first column
in Table 2. It shows not only the orientability of the manifold, but also
that the Fisher metric is a generic Riemannian metric.
\end{rem}
Since some preparation is needed to prove the theorem, we first start
with several lemmas.
\begin{lem}
$N^{d}$ is simply connected for all $d\in\mathbb{N}$.\end{lem}
\begin{proof}
As a topological space, $N^{d}$ is homeomorphic to the parameter
space $\mathbb{R}^{d}\times PD(d,\mathbb{R})\subset\mathbb{R}^{\frac{d(d+3)}{2}}$,
as we stated in Remark 5.2. The $\mathbb{R}^{d}$ part is contractible
showing $\pi_{1}(\mathbb{R}^{d})=0$. By the theory of linear algebra,
we see that if $A,B\in PD(d,\mathbb{R})$, then $(1-t)A+tB\in PD(d,\mathbb{R})$,
$\forall t\in[0,1]$. Therefore, the space $PD(d,\mathbb{R})$ is
convex, hence also contractible. In particular, $\pi_{1}(PD(d,\mathbb{R}))=0$.
By the theory of algebraic topology, we have
\[
\pi_{1}(N^{d})\cong\pi_{1}(\mathbb{R}^{d}\times PD(d,\mathbb{R}))\cong\pi_{1}(\mathbb{R}^{d})\times\pi_{1}(PD(d,\mathbb{R}))=0
\]
as desired.\end{proof}
\begin{cor}
$H_{d}=H_{d}^{0}$. \end{cor}
\begin{proof}
This result is straightforward by applying Lemma 5.5 and part (3)
of Proposition 3.3.\end{proof}
\begin{lem}
$N^{1}$ is isometric to the 2-dimensional hyperbolic space, denoted
by $H(2)$.\end{lem}
\begin{proof}
By definition, we have
\[
N^{1}=\left\{ p(x,\mu,\sigma)=\frac{1}{\sqrt{2\pi}\sigma}\exp\left\{-\frac{(x-\mu)^{2}}{2\sigma^{2}}\right\}\middle|\mu\in\mathbb{R},\sigma\in\mathbb{R}_{+}\right\} .
\]
Let $\theta_{1}=\frac{\mu}{\sigma^{2}}$ and $\theta_{2}=-\frac{1}{2\sigma^{2}}$
be the natural coordinates of the exponential family. Then the Fisher
metric matrix is given by
\[
\begin{bmatrix}g_{ij}\end{bmatrix}=\begin{bmatrix}\sigma^{2} & 2\mu\sigma^{2}\\
2\mu\sigma^{2} & 2\sigma^{2}(2\mu^{2}+\sigma^{2})
\end{bmatrix},
\]
and the curvature tensor follows $R_{1212}=\frac{1}{\sigma^{6}}$.

As a result, the sectional curvature(also the Gaussian curvature)
is $\kappa=-\frac{1}{2}$, which is a negative constant. Thus, $N^{1}$
is a complete simply connected manifold with constant sectional curvature
$-\frac{1}{2}$, hence is the space form and isometric to $2$-dimensional
hyperbolic space $H(2)$ with
\[
\dim N^{1}=\frac{1\left(1+3\right)}{2}=2.
\]
\end{proof}
\begin{lem}
The $n$-dimensional hyperbolic space $H(n)$ is a symmetric space
for all $n\in\mathbb{Z}_{+}.$\end{lem}
\begin{proof}
Consider $s=\begin{bmatrix}I_{n} & 0\\
0 & -1
\end{bmatrix}\in M(n+1,\mathbb{R})$ where $I_{n}$ is the $n\times n$ identity matrix and
\[
O(n,1):=\{A\in GL(n+1,\mathbb{R})|A^{T}sA=s\}.
\]
Here, $O(n,1)$ is called the Lorentz group consisting of all linear
transformation on $\mathbb{R}^{n+1}$ maintaining the invariance of
the Lorentz inner product, which is defined by
\[
\langle X,Y\rangle_{L}:=\sum\limits _{i=1}^{n}x^{i}y^{i}-x^{n+1}y^{n+1}=X^{T}sY,
\]
$X=(x^{1},\dots,x^{n+1}),Y=(y^{1},\dots,y^{n+1})\in\mathbb{R}^{n+1}$.
Note that $O(n,1)$ has $4$ components and the one containing $I$
is
\[
G=\{A=(a_{ij})\in O(n,1)|\det A=1,a_{\left(n+1\right)\left(n+1\right)}\geqslant1\}
,\]
which is a connected Lie group and acts on the Lorentz space $\mathbb{R}_{L}^{n+1}=(\mathbb{R}^{n+1},\langle\cdot,\cdot\rangle_{L})$
keeping $H(n)=\{X=(x^{1},\dots,x^{n+1})^{T}\in\mathbb{R}^{n+1}|\langle X,X\rangle_{L}=-1,\ x^{n+1}>0\}$
invariant. The Lorentz inner product induces a Riemannian metric $g$
on $H(n)$. Also,
\begin{eqnarray*}
\sigma:G & \longrightarrow & G\\
A & \longmapsto & sAs
\end{eqnarray*}
is an involution automorphism on $G$. Note that the fixed point subgroup
\begin{eqnarray*}
K_{\sigma} & = & \{A\in G|\sigma(A)=A\}=G\cap O(n+1)\\
 & = & \left\{ A=\begin{bmatrix}B & 0\\
0 & 1
\end{bmatrix}\middle|B\in SO(n)\right\} \cong SO(n).
\end{eqnarray*}
Thus, $K_{\sigma}$ is a compact connected Lie group, which means
that $(G,K_{\sigma},\sigma)$ is a Riemannian symmetric pair, and
$H(n)=G/K_{\sigma}$ is a Riemannian symmetric space.

In fact, $H(n)$ is just of the type BD\mbox{I} in Table 1 with $p=1$
and $q=n-1$.\end{proof}
\begin{prop}
$H_{1}=SO(2)$. \end{prop}
\begin{proof}
It follows from Lemma 5.7, Lemma 5.8 and Corollary 4.3.\end{proof}
\begin{lem}
Let $(M^n,g)$ be a Riemannian manifold, if M is isometric to $M_1^{n_1}\times M_2^{n_2}$, where $n_1+n_2=n$. Suppose matrix $K_{n\times n}$ is the sectional curvature matrix of M, then $K$ must be block diagonal. All geometric structures are presented in appendix
\end{lem}
\begin{rem}
Lemma 5.10 is elementary in theory of submanifold, see also [44].
\end{rem}
\begin{lem}
$N^{d}$ is irreducible for $d=1,2,3$. \end{lem}
\begin{proof}
Lemma 5.7 implies that $N^1$ is irreducible. When d=2,3, a direct computation of sectional curvature shows that the sectional curvature matrix for $N^d$ is not block diagonal, hence irreducible (see appendix).
\end{proof}
\begin{rem}
In general, the covariance matrix $\Sigma$ is not necessarily diagonal, hence $N^{d}$ is very likely to be irreducible, for all $d\in\mathbb{N}$.
\end{rem}
\begin{lem}
$N^{d}$ is not symmetric for all $d\ge2$.\end{lem}
\begin{proof}
When $d=2,3$, appendix shows that $\kappa$ is neither negative nor positive(of course not identically
$0$ either). We conclude that $N^{2}$ and $N^3$ are not of Euclidean type,
compact type or non-compact type as well, hence not a symmetric space.

\end{proof}
\begin{cor}
$H^{d}$must lie in the following Berger's list for $d=2,3$.

\begin{tabular}{|c|c|c|c|c|c|c|}
\hline
{\small $SO\left(\frac{d\left(d+3\right)}{2}\right)$} & {\small $U\left(\frac{d\left(d+3\right)}{4}\right)$} & {\small $SU\left(\frac{d\left(d+3\right)}{4}\right)$} & {\small $Sp\left(\frac{d\left(d+3\right)}{8}\right)\cdot Sp\left(1\right)$} & {\small $Sp\left(\frac{d\left(d+3\right)}{8}\right)$} & {\tiny $G_{2}$} & {\small $Spin\left(7\right)$}\tabularnewline
\hline
\end{tabular}.\end{cor}
\begin{proof}
$H^{d}$ is simply connected by Lemma 5.5, irreducible by Lemma 5.12
and also nonsymmetric by Lemma 5.14. Hence Theorem 4.5(M. Berger) applies.
\end{proof}
\begin{lem}
$H^{d}$ is neither $G_{2}$ nor $Spin(7)$.
\end{lem}
\begin{proof}
$dim(N^{d})=\frac{d(d+3)}{2}=2,5,9,14,20,\dots$. However, Berger's
list indicates that every manifold with $G_{2}$ or $Spin(7)$ as
its holonomy group must be of dimension $7$ or $8$, respectively.
As a result, $H^{d}$ is neither $G_{2}$ nor $Spin(7)$.
\end{proof}
\begin{lem}
$H^{d}$ is not equal to any of the following groups
\[
SU(\frac{d(d+3)}{4}),\ Sp(\frac{d(d+3)}{8})\cdot Sp(1),\ Sp(\frac{d(d+3)}{8}).
\]
\end{lem}
\begin{proof}
By the comments in Berger's list we see that a manifold possesses
holonomy groups $SU(\frac{d(d+3)}{4})$, $Sp(\frac{d(d+3)}{8})\cdot Sp(1)$
or $Sp(\frac{d(d+3)}{8})$ must be an Einstein manifold, namely, there
exists a constant $k$, s.t.,
\[
Ric=kg.
\]
However, it is obvious that $N^{2}$ and $N^3$ are not Einstein manifolds (see appendix), 
which implies that $H^{d}$ is not equal to any of the following groups:
\[
SU(\frac{d(d+3)}{4}),\ Sp(\frac{d(d+3)}{8})\cdot Sp(1),\ Sp(\frac{d(d+3)}{8}).
\]
\end{proof}
\begin{lem}
$N^{d}$ is not K\"{a}hlerian for all $d\in\mathbb{N}$. \end{lem}
\begin{proof}
Takano(\cite{key-39}-\cite{key-43}) has proved that $(N^{d},\nabla^{(\alpha)})$
admits an almost complex structure $J^{(\alpha)}$ that is parallel
to the $\alpha$-connection $\nabla^{(\alpha)}$ only if $\alpha=\pm1$.
Recall Proposition 2.4 that $\nabla^{(\alpha)}$ is the Levi-Civita
connection if and only if $\alpha=0$. Hence $N^{d}$ does not admit
a K\"{a}hler metric.
\end{proof}
A direct corollary follows.
\begin{cor}
$H^{d}\neq U(\frac{d(d+3)}{4})$.
\end{cor}
Based on all preparations, we could show the proof of theorem 5.3.
\begin{proof}
When $d=1$, Proposition 5.9 proves the case.

When $d\geqslant2$ we get a possible list in corollary 5.15. Furthermore, Lemma
5.16, Lemma 5.17 and Corollary 5.19 rule out all possible groups except
for $SO(\frac{d(d+3)}{2})$, as desired.

Thus, to sum up, we could conclude that
\[
H^{d}=H_{0}^{d}=SO(\frac{d(d+3)}{2}),\ d=1,2,3.
\]

\end{proof}
In fact, part of our results about the normal distribution manifolds
can be generalized to the exponential family.

\section{Holonomy Group of Exponential Family}

Let $S$ be an exponential family with dimension $n$ and $H$ be
its holonomy group.
\begin{lem}
$S$ is not K\"{a}hlerian. \end{lem}
\begin{proof}
Similarly as that in lemma 5.18, $S$ admits an almost complex structure
$J^{(\alpha)}$ that is parallel to the $\alpha$-connection $\nabla^{(\alpha)}$
only if $\alpha=\pm1$, which is not the Levi-Civita connection(\cite{key-39}-\cite{key-43}). \end{proof}
\begin{cor}
$H$ is not equal to any of following groups
\[
U(\frac{n}{2}),\ SU(\frac{n}{2}),\ Sp(\frac{n}{4}).
\]
 \end{cor}
\begin{proof}
Lemma 6.1 implies that $H$ is not a subgroup of $U(\frac{n}{2})$.
Noting that
\[
Sp(\frac{n}{4})<SU(\frac{n}{2})<U(\frac{n}{2}),
\]
hence $H$ cannot be any of them. \end{proof}

Now, after ruling out several cases, the following theorem holds

\begin{thm}
If $S$ is a simply connected nonsymmetric $n$-dimensional exponential
family with irreducible holonomy group $H$, then $H$ must be one of the following
groups

\begin{center}
\begin{tabular}{|c|c|}
\hline
Holonomy & Dimension\tabularnewline
\hline
$SO\left(m\right)$ & $n=m$\tabularnewline
\hline
$Sp\left(m\right)\cdot SP\left(1\right)$ & $n=4m$\tabularnewline
\hline
$G_{2}$ & $n=7$\tabularnewline
\hline
$Spin\left(7\right)$ & $n=8$\tabularnewline
\hline
\end{tabular}.
\end{center}
\end{thm}
\begin{cor}
Suppose $S$ is a simply connected nonsymmetric n-dimensional exponential
family with irreducible holonomy group $H$. We have

(1) if $n\neq7,8$, then $H$ is either $SO(n)$ or $Sp(m)\cdot Sp(1)$,
where $n=4m$;

(2) if $n=7$, then $H$ is either $SO(7)$ or $G_{2}$;

(3) if $n\neq7$ and $n$ is odd, then $H=SO(n)$;

(4) if $n=2(2m+1)$, then $H=SO(n)$;

(5) if $n=8$, then $H=SO(8)$, or $Sp(2)\cdot Sp(1)$ or $Spin(7)$;

(6) if $n=4m$ where $m\neq2$, then $H$ is either $SO(n)$ or $Sp(m)\cdot Sp(1)$;

(7) if $S$ is not an Einstein manifold, then $H=SO(n)$.
\end{cor}
\begin{proof}
(1)-(6) directly follow from Theorem 6.3, hence the remaining is to prove
(7). If $H$ is a subgroup of $G_{2}$ or $Spin(7)$, then the Ricci
curvature must be identically $0$(\cite{key-12}), implying $S$
is an Einstein manifold, which is a contradiction.
\end{proof}
Since almost all common examples of exponential families are not Einstein,
the holonomy groups of almost all exponential families are $SO(n)$.
There is only one exception, the monistic normal distribution manifold $N^{1}$, which is Einstein.
However, $H^{1}=SO(2)=SO\left(\dim N^{1}\right)$ (Proposition 5.9) which coincides
with our results.

\section{Conclusion}

After some preliminaries about information geometry and holonomy groups,
two main results, Theorem 5.3 and Theorem 6.3 are proved. Theorem
5.3 shows that the holonomy groups of normal distribution manifolds
are special orthogonal groups for all dimensions. In addition, a list
of possible holonomy groups for general exponential families is presented
in Theorem 6.3.

\section{Appendix}

We presents two tables mentioned in Section 4 on classification of Riemannian holonomy groups and the geometric structure of $N^d$ used in Section 5 here.

\begin{table}[!htbp]
\begin{center}
\caption{\label{comparison} One of four lists of Riemannian symmetric spaces}
\begin{tabular}{|c|c|c|c|c|p{4.3cm}|}
  \hline
  Label & G & K & Dimension & Rank & Geometric interpretation \\
  \hline
  A\uppercase\expandafter{\romannumeral1}  & $SL(n;\mathbb R)$ & $SO(n)$ & $\frac{(n-1)(n+2)}{2}$ & n-1 & Set of $\mathbb{R}P^{n}_{hyp}$'s in $\mathbb{C}P^{n}_{hyp}$\\
  \hline
  A\uppercase\expandafter{\romannumeral2}  & $SL(n;\mathbb H)$ & $Sp(n)$ & $(n-1)(2n+1)$ &n-1& Set of $\mathbb{H}P^{n-1}_{hyp}$'s in $\mathbb{C}P^{2n-1}_{hyp}$\\
  \hline
  A\uppercase\expandafter{\romannumeral3}  & $SU(p,q)$ & $S(U(p)\times U(q))$ & $2pq$ &$\min(p,q)$& $G_p(p,q;\mathbb{C})$ \\
  \hline
  BD\uppercase\expandafter{\romannumeral1}  & $SO_0(p,q)$ & $SO(p)\times SO(q)$ & $pq$ &$\min(p,q)$ & $G_p(p,q;\mathbb{R})$\\
  \hline
  D\uppercase\expandafter{\romannumeral3} & $SO(n;\mathbb H)$ & $U(n)$ & $n(n-1)$ &$[\frac{n}{2}]$& Set of $\mathbb{C}P^{n-1}_{hyp}$'s in $\mathbb{R}P^{2n-1}_{hyp}$\\
  \hline
  C\uppercase\expandafter{\romannumeral1} & $Sp(n;\mathbb R)$ & $U(n)$ & $n(n+1)$ &n& Set of $\mathbb{C}P_{hyp}^{n}$'s in $\mathbb{H}P_{hyp}^{n}$ \\
  \hline
  C\uppercase\expandafter{\romannumeral2} & $Sp(p,q)$  & $Sp(p)\times Sp(q)$ & $4pq$ &$\min(p,q)$&  $G_p(p,q;\mathbb{H})$\\
  \hline
  E\uppercase\expandafter{\romannumeral1} & $E_6^6$ & $Sp(4)$ & 42 &6& Antichains of $(\mathbb{C}\otimes\mathbb O)P_{hyp}^2$\\
  \hline
  E\uppercase\expandafter{\romannumeral2} & $E_6^2$ & $SU(6)\times SU(2)$ & 40 &4& Set of the $(\mathbb{C}\otimes\mathbb{H})P_{hyp}^2$'s in $(\mathbb{C}\otimes\mathbb O)P_{hyp}^2$\\
  \hline
  E\uppercase\expandafter{\romannumeral3} & $E_6^{-14}$ & $SO(10)\times SO(2)$ & 32 &2& Rosenfeld's hyperbolic projective plane $(\mathbb{C}\otimes\mathbb O)P_{hyp}^2$ \\
  \hline
  E\uppercase\expandafter{\romannumeral4} & $E_6^{-26}$ & $F_4$ & 26 &2& Set of $\mathbb{O}P^2$'s the in $(\mathbb{C}\otimes\mathbb O)P_{hyp}^2$ \\
  \hline
  E\uppercase\expandafter{\romannumeral5} & $E_7^7$ & $SU(8)$ & 70 &7& Antichains of $(\mathbb{H}\otimes\mathbb O)P_{hyp}^2$ \\
  \hline
  E\uppercase\expandafter{\romannumeral6} & $E_7^{-5}$ & $SO(12)\times SU(2)$ & 64 &4& Rosenfeld hyperbolic projective plane $(\mathbb H\otimes\mathbb O)P_{hyp}^2$  \\
  \hline
  E\uppercase\expandafter{\romannumeral7} & $E_7^{-25}$ & $E_6\times SO(2)$ & 54 &3& Set of the $(\mathbb{C}\otimes\mathbb{O})P_{hyp}^2$'s in $(\mathbb{H}\otimes\mathbb O)P_{hyp}^2$\\
  \hline
  E\uppercase\expandafter{\romannumeral8} & $E_8^8$ & $SO(16)$ & 128 &8& Rosenfeld projective plane $(\mathbb O\otimes\mathbb O)P_{hyp}^2$ \\
  \hline
  E\uppercase\expandafter{\romannumeral9} & $E_8^{-24}$ & $E_7\times SU(2)$ & 112 &4& Set of the $(\mathbb{H}\otimes\mathbb{O})P_{hyp}^2$'s in $(\mathbb{O}\otimes\mathbb O)P_{hyp}^2$ \\
  \hline
  F\uppercase\expandafter{\romannumeral1} & $F_4^4$ & $Sp(3)\times SU(2)$ & 28 &4& Set of the $\mathbb{H}P_{hyp}^2$'s in $\mathbb OP_{hyp}^2$ \\
  \hline
  F\uppercase\expandafter{\romannumeral2} & $F_4^{-20}$ & $SO(9)$ & 16 &1& Hyperbolic Cayley projective plane $\mathbb{O}P_{hyp}^2$ \\
  \hline
  G & $G_2^2$ & $SU(2)\times SU(2)$ & 8 &2& Set of non-division $\mathbb{H}$ subalgebras of the non-division $\mathbb{O}$\\
  \hline
\end{tabular}
\end{center}
\end{table}

\begin{table}
\begin{center}
\caption{\label{comparison} List of Riemannian holonomy groups}
\begin{tabular}{|c|c|c|c|}
  \hline
  $H$ & Dimension & Type of manifold & Comments \\
  \hline
  $SO(n)$ & $n$ & Oriented manifold & Generic Metric \\
  \hline
  $U(n)$ & $2n$ & K$\ddot{a}$hler manifold & K$\ddot{a}$hler \\
  \hline
  $SU(n)$ & $2n$ & Calabi-Yau manifold & Ricci-flat, K$\ddot{a}$hler \\
  \hline
  $Sp(n)\cdot Sp(1)$ & $4n$ & Quaternion-K$\ddot{a}$hler manifold & Eistein \\
  \hline
  $Sp(n)$ & $4n$ & Hyperk$\ddot{a}$hler manifold & Ricci-flat, K$\ddot{a}$hler \\
  \hline
  $G_2$ & $7$ & $G_2$ manifold & Ricci-flat \\
  \hline
  $Spin(7)$ & $8$ & $Spin(7)$ manifold& Ricci-flat \\
  \hline
\end{tabular}
\end{center}
\end{table}

$$N^2=\{p(x,y,\mu_1,\mu_2,\sigma_1,\sigma_2,\sigma_{12})=\frac{1}{2\pi\sqrt{\sigma_1\sigma_2-\sigma_{12}^2}}\exp\{-AB\}|\mu_1,\mu_2\in\mathbb{R},\sigma_1,\sigma_2\in\mathbb{R}_+,\sigma_{12}=cov(X,Y)\},$$
where $A=\frac{1}{2(\sigma_2\sigma_2-\sigma_{12}^2)}$, $B=\sigma_2(x-\mu_1)^2-2\sigma_{12}(x-\mu_1)(y-\mu_2)+\sigma_1(y-\mu_2)^2$.
This is a $5$-dimensional manifold. It is obvious that 
\begin{align*}
p(x,y)&=\log(\frac{1}{2\pi\sqrt{\Delta}}e^{-\frac{1}{2\Delta}(\sigma_2(x-\mu_1)^2-2\sigma_{12}(x-\mu_1)(y-\mu_2)+\sigma_1(y-\mu_2)^2)})\\
&=\frac{\mu_1\sigma_2-\mu_2\sigma_{12}}{\Delta}x+\frac{\mu_2\sigma_1-\mu_1\sigma_{12}}{\Delta}y+\frac{-\sigma_2}{2\Delta}x^2+\frac{\sigma_{12}}{2\Delta}xy+\frac{-\sigma_1}{2\Delta}y^2\\
&-(\log(2\pi\sqrt{\Delta})+\frac{{\mu_2}^2\sigma_1+{\mu_1}^2\sigma_2-2\mu_1\mu_2\sigma_{12}}{2\Delta}),
\end{align*}
where $\Delta=\sigma_1\sigma_2-\sigma_{12}^2$.
Hence, the coordinates for exponential family are
$$\theta_1=\frac{\mu_1\sigma_2-\mu_2\sigma_{12}}{\Delta},\theta_2=\frac{\mu_2\sigma_1-\mu_1\sigma_{12}}{\Delta},\theta_3=
-\frac{\sigma_2}{2\Delta},\theta_4=-\frac{\sigma_{12}}{\Delta},\theta_5=-\frac{\sigma_1}{2\Delta},$$
and $\Delta=\frac{1}{4\theta_3\theta_5-{\theta_4}^2}$. While the potential function is 
\begin{align*}
\phi(\theta)&=\log(2\pi\sqrt{\Delta})+\frac{{\mu_2}^2\sigma_1+{\mu_1}^2\sigma_2-2\mu_1\mu_2\sigma_{12}}{2\Delta}\\
&=\log(2\pi\sqrt{\Delta})-\Delta({\theta_2}^2\theta_3-\theta_1\theta_2\theta_4+{\theta_1}^2\theta_5).
\end{align*}
The components $g_{ij}$, $i,j=1,2,3,4,5$ of the Fisher metric matrix are given by \\
$g_{11}=\sigma_{{1}}$,\\
$g_{12}= \sigma_{{12}}$,\\
$g_{13}=2                                \mu_{{1
}}\sigma_{{1}}$,\\
$g_{14}=\sigma_{{1}}\mu_{{2}}+\sigma_{{12}}\mu_{{1}}$,\\
$g_{15}=2\sigma_{{12}}\mu_{{2}}$,\\
$g_{22}=\sigma_{{2
2}}$,\\
$g_{23}=2\sigma_{{12}}\mu_{{1}}$,\\
$g_{24}=\sigma_{{2}}\mu_{{1}}+\sigma_{{12}}
\mu_{{2}}$,\\
$g_{25}=2\sigma_{{2}}\mu_{{2}} $,\\
$g_{33}=2\sigma_{{1}} \left(
\sigma_{{1}}+2{\mu_{{1}}}^{2} \right)$,\\
$g_{34}=2\sigma_{{12}}\sigma_{{1
1}}+2\mu_{{1}}\sigma_{{1}}\mu_{{2}}+2{\mu_{{1}}}^{2}\sigma_{{1
2}}$,\\
$g_{35}=2\sigma_{{12}} \left( \sigma_{{12}}+2\mu_{{1}}\mu_{{2}}
 \right)$,\\
$g_{44}=\sigma_{{1}}\sigma_{{2}}+\sigma_{{1}}{\mu_{{2
}}}^{2}+{\mu_{{1}}}^{2}\sigma_{{2}}+2\mu_{{1}}\sigma_{{12}}\mu_{{
2}}+{\sigma_{{12}}}^{2}$,\\
$g_{45}=2\sigma_{{2}}\sigma_{{12}}+2\sigma_{{2
2}}\mu_{{1}}\mu_{{2}}+2\sigma_{{12}}{\mu_{{2}}}^{2}$,\\
$g_{55}=2\sigma_{{2}} \left(
\sigma_{{2}}+2{\mu_{{2}}}^{2}\right) $.\\

The components $K_{ij}$, $i,j=1,2,3,4,5$ of the sectional curvature are given by\\
$K_{ii}=0$, $i=1,2,3,4,5$,\\
$K_{12}=1/4-1/4{\alpha}^{2}$,\\
$K_{13}=-1/2+1/2{\alpha}^{2}$,\\
$K_{14}=-1/4{\frac { \left( -3{\sigma_{{2}}}^{2}\sigma_{{1}}-{\sigma_{
{1}}}^{2}\sigma_{{22}}-{\sigma_{{2}}}^{2}{\mu_{{1}}}^{2}+\sigma_{
{1}}{\mu_{{1}}}^{2}\sigma_{{22}} \right) {\alpha}^{2}}{{\sigma_{{1
1}}}^{2}\sigma_{{22}}+\sigma_{{1}}{\mu_{{1}}}^{2}\sigma_{{22}}+{
\sigma_{{2}}}^{2}\sigma_{{1}}-{\sigma_{{2}}}^{2}{\mu_{{1}}}^{2}}
}-1/4{\frac {{\sigma_{{1}}}^{2}\sigma_{{22}}+3{\sigma_{{2}}}^
{2}\sigma_{{1}}-\sigma_{{1}}{\mu_{{1}}}^{2}\sigma_{{22}}+{\sigma_
{{2}}}^{2}{\mu_{{1}}}^{2}}{{\sigma_{{1}}}^{2}\sigma_{{22}}+\sigma
_{{1}}{\mu_{{1}}}^{2}\sigma_{{22}}+{\sigma_{{2}}}^{2}\sigma_{{1
}}-{\sigma_{{2}}}^{2}{\mu_{{1}}}^{2}}}$,\\
$K_{15}=-1/2{\frac { \left( \sigma_{{1}}{\mu_{{2}}}^{2}\sigma_{{22}}-
\sigma_{{22}}{\sigma_{{2}}}^{2}-{\sigma_{{2}}}^{2}{\mu_{{2}}}^{2}
 \right) {\alpha}^{2}}{{\sigma_{{22}}}^{2}\sigma_{{1}}+2\sigma_{{
1}}{\mu_{{2}}}^{2}\sigma_{{22}}-2{\sigma_{{2}}}^{2}{\mu_{{2}}}^
{2}}}-1/2{\frac {-\sigma_{{1}}{\mu_{{2}}}^{2}\sigma_{{22}}+\sigma
_{{22}}{\sigma_{{2}}}^{2}+{\sigma_{{2}}}^{2}{\mu_{{2}}}^{2}}{{
\sigma_{{22}}}^{2}\sigma_{{1}}+2\sigma_{{1}}{\mu_{{2}}}^{2}
\sigma_{{22}}-2{\sigma_{{2}}}^{2}{\mu_{{2}}}^{2}}}$,\\
$K_{23}=-1/2{\frac { \left( -{\sigma_{{2}}}^{2}\sigma_{{1}}-{\sigma_{{1
2}}}^{2}{\mu_{{1}}}^{2}+\sigma_{{1}}{\mu_{{1}}}^{2}\sigma_{{22}}
 \right) {\alpha}^{2}}{{\sigma_{{1}}}^{2}\sigma_{{22}}+2\sigma_{{
1}}{\mu_{{1}}}^{2}\sigma_{{22}}-2{\sigma_{{2}}}^{2}{\mu_{{1}}}^
{2}}}-1/2{\frac {-\sigma_{{1}}{\mu_{{1}}}^{2}\sigma_{{22}}+{
\sigma_{{2}}}^{2}{\mu_{{1}}}^{2}+{\sigma_{{2}}}^{2}\sigma_{{1}}}
{{\sigma_{{1}}}^{2}\sigma_{{22}}+2\sigma_{{1}}{\mu_{{1}}}^{2}
\sigma_{{22}}-2{\sigma_{{2}}}^{2}{\mu_{{1}}}^{2}}}$,\\
$K_{24}=-1/4{\frac { \left( -{\sigma_{{22}}}^{2}\sigma_{{1}}+\sigma_{{1
}}{\mu_{{2}}}^{2}\sigma_{{22}}-3\sigma_{{22}}{\sigma_{{2}}}^{2}-
{\sigma_{{2}}}^{2}{\mu_{{2}}}^{2} \right) {\alpha}^{2}}{{\sigma_{{2
2}}}^{2}\sigma_{{1}}+\sigma_{{1}}{\mu_{{2}}}^{2}\sigma_{{22}}+
\sigma_{{22}}{\sigma_{{2}}}^{2}-{\sigma_{{2}}}^{2}{\mu_{{2}}}^{2}
}}-1/4{\frac {{\sigma_{{22}}}^{2}\sigma_{{1}}-\sigma_{{1}}{\mu_
{{2}}}^{2}\sigma_{{22}}+3\sigma_{{22}}{\sigma_{{2}}}^{2}+{\sigma
_{{2}}}^{2}{\mu_{{2}}}^{2}}{{\sigma_{{22}}}^{2}\sigma_{{1}}+
\sigma_{{1}}{\mu_{{2}}}^{2}\sigma_{{22}}+\sigma_{{22}}{\sigma_{{1
2}}}^{2}-{\sigma_{{2}}}^{2}{\mu_{{2}}}^{2}}}$,\\
$K_{25}=-1/2+1/2{\alpha}^{2}$,\\
\begin{align*}
K_{34}&=-1/2{\frac { \left( -{\sigma_{{1}}}^{3}\sigma_{{22}}-{\sigma_{{1
1}}}^{3}{\mu_{{2}}}^{2}-3{\sigma_{{1}}}^{2}{\mu_{{1}}}^{2}\sigma_{
{22}}+{\sigma_{{2}}}^{2}{\sigma_{{1}}}^{2}+2{\sigma_{{1}}}^{2
}\mu_{{1}}\sigma_{{2}}\mu_{{2}}+\sigma_{{1}}\sigma_{{22}}{\mu_{{1
}}}^{4}+2\sigma_{{1}}{\sigma_{{2}}}^{2}{\mu_{{1}}}^{2}-{\sigma_{
{2}}}^{2}{\mu_{{1}}}^{4} \right) {\alpha}^{2}}{{\sigma_{{1}}}^{3}
\sigma_{{22}}+{\sigma_{{1}}}^{3}{\mu_{{2}}}^{2}+3{\sigma_{{1}}}
^{2}{\mu_{{1}}}^{2}\sigma_{{22}}-2{\sigma_{{1}}}^{2}\mu_{{1}}
\sigma_{{2}}\mu_{{2}}-{\sigma_{{2}}}^{2}{\sigma_{{1}}}^{2}+2
\sigma_{{1}}\sigma_{{22}}{\mu_{{1}}}^{4}-2\sigma_{{1}}{\sigma_{
{2}}}^{2}{\mu_{{1}}}^{2}-2{\sigma_{{2}}}^{2}{\mu_{{1}}}^{4}}}\\
&\ \ \ -1/2{\frac {{\sigma_{{1}}}^{3}\sigma_{{22}}+{\sigma_{{2}}}^{2}{\mu
_{{1}}}^{4}-{\sigma_{{2}}}^{2}{\sigma_{{1}}}^{2}+{\sigma_{{1}}}^
{3}{\mu_{{2}}}^{2}+3{\sigma_{{1}}}^{2}{\mu_{{1}}}^{2}\sigma_{{22}
}-2{\sigma_{{1}}}^{2}\mu_{{1}}\sigma_{{2}}\mu_{{2}}-\sigma_{{1
}}\sigma_{{22}}{\mu_{{1}}}^{4}-2\sigma_{{1}}{\sigma_{{2}}}^{2}{
\mu_{{1}}}^{2}}{{\sigma_{{1}}}^{3}\sigma_{{22}}+{\sigma_{{1}}}^{3
}{\mu_{{2}}}^{2}+3{\sigma_{{1}}}^{2}{\mu_{{1}}}^{2}\sigma_{{22}}-
2{\sigma_{{1}}}^{2}\mu_{{1}}\sigma_{{2}}\mu_{{2}}-{\sigma_{{2}
}}^{2}{\sigma_{{1}}}^{2}+2\sigma_{{1}}\sigma_{{22}}{\mu_{{1}}}^
{4}-2\sigma_{{1}}{\sigma_{{2}}}^{2}{\mu_{{1}}}^{2}-2{\sigma_{{
2}}}^{2}{\mu_{{1}}}^{4}}},
\end{align*}
\begin{align*}
K_{35}&=-{\frac { \left( -\sigma_{{1}}\sigma_{{22}}{\sigma_{{2}}}^{2}-2
\sigma_{{1}}\sigma_{{22}}\mu_{{1}}\sigma_{{2}}\mu_{{2}}+\sigma_{{
1}}\sigma_{{22}}{\mu_{{1}}}^{2}{\mu_{{2}}}^{2}-\sigma_{{1}}{
\sigma_{{2}}}^{2}{\mu_{{2}}}^{2}-{\sigma_{{2}}}^{2}{\mu_{{1}}}^{2}
\sigma_{{22}}-{\mu_{{1}}}^{2}{\sigma_{{2}}}^{2}{\mu_{{2}}}^{2}+{
\sigma_{{2}}}^{4}+4{\sigma_{{2}}}^{3}\mu_{{1}}\mu_{{2}} \right)
{\alpha}^{2}}{{\sigma_{{1}}}^{2}{\sigma_{{22}}}^{2}+2{\sigma_{{1
1}}}^{2}{\mu_{{2}}}^{2}\sigma_{{22}}+2\sigma_{{1}}{\sigma_{{22}}
}^{2}{\mu_{{1}}}^{2}+4\sigma_{{1}}\sigma_{{22}}{\mu_{{1}}}^{2}{
\mu_{{2}}}^{2}-4{\sigma_{{2}}}^{3}\mu_{{1}}\mu_{{2}}-4{\mu_{{1}}
}^{2}{\sigma_{{2}}}^{2}{\mu_{{2}}}^{2}-{\sigma_{{2}}}^{4}}}\\
&\ \ \ -{
\frac {\sigma_{{1}}\sigma_{{22}}{\sigma_{{2}}}^{2}+2\sigma_{{1
1}}\sigma_{{22}}\mu_{{1}}\sigma_{{2}}\mu_{{2}}-\sigma_{{1}}\sigma
_{{22}}{\mu_{{1}}}^{2}{\mu_{{2}}}^{2}+\sigma_{{1}}{\sigma_{{2}}}^
{2}{\mu_{{2}}}^{2}+{\sigma_{{2}}}^{2}{\mu_{{1}}}^{2}\sigma_{{22}}+{
\mu_{{1}}}^{2}{\sigma_{{2}}}^{2}{\mu_{{2}}}^{2}-{\sigma_{{2}}}^{4}
-4{\sigma_{{2}}}^{3}\mu_{{1}}\mu_{{2}}}{{\sigma_{{1}}}^{2}{
\sigma_{{22}}}^{2}+2{\sigma_{{1}}}^{2}{\mu_{{2}}}^{2}\sigma_{{22
}}+2\sigma_{{1}}{\sigma_{{22}}}^{2}{\mu_{{1}}}^{2}+4\sigma_{{1
1}}\sigma_{{22}}{\mu_{{1}}}^{2}{\mu_{{2}}}^{2}-4{\sigma_{{2}}}^{3
}\mu_{{1}}\mu_{{2}}-4{\mu_{{1}}}^{2}{\sigma_{{2}}}^{2}{\mu_{{2}}}^
{2}-{\sigma_{{2}}}^{4}}},
\end{align*}
\begin{align*}
K_{45}&=1/2{\frac { \left( \sigma_{{1}}{\sigma_{{22}}}^{3}+3\sigma_{{1
1}}{\sigma_{{22}}}^{2}{\mu_{{2}}}^{2}-\sigma_{{1}}\sigma_{{22}}{
\mu_{{2}}}^{4}+{\sigma_{{22}}}^{3}{\mu_{{1}}}^{2}-{\sigma_{{22}}}^{2
}{\sigma_{{2}}}^{2}-2{\sigma_{{22}}}^{2}\mu_{{1}}\sigma_{{2}}
\mu_{{2}}-2\sigma_{{22}}{\sigma_{{2}}}^{2}{\mu_{{2}}}^{2}+{\sigma
_{{2}}}^{2}{\mu_{{2}}}^{4} \right) {\alpha}^{2}}{\sigma_{{1}}{
\sigma_{{22}}}^{3}+3\sigma_{{1}}{\sigma_{{22}}}^{2}{\mu_{{2}}}^{
2}+2\sigma_{{1}}\sigma_{{22}}{\mu_{{2}}}^{4}+{\sigma_{{22}}}^{3}
{\mu_{{1}}}^{2}-2{\sigma_{{22}}}^{2}\mu_{{1}}\sigma_{{2}}\mu_{{2}
}-{\sigma_{{22}}}^{2}{\sigma_{{2}}}^{2}-2\sigma_{{22}}{\sigma_{{
2}}}^{2}{\mu_{{2}}}^{2}-2{\sigma_{{2}}}^{2}{\mu_{{2}}}^{4}}}\\
&\ \ \ +1/2
{\frac {-\sigma_{{1}}{\sigma_{{22}}}^{3}-3\sigma_{{1}}{\sigma
_{{22}}}^{2}{\mu_{{2}}}^{2}+\sigma_{{1}}\sigma_{{22}}{\mu_{{2}}}^{
4}-{\sigma_{{22}}}^{3}{\mu_{{1}}}^{2}+{\sigma_{{22}}}^{2}{\sigma_{{1
2}}}^{2}+2{\sigma_{{22}}}^{2}\mu_{{1}}\sigma_{{2}}\mu_{{2}}+2
\sigma_{{22}}{\sigma_{{2}}}^{2}{\mu_{{2}}}^{2}-{\sigma_{{2}}}^{2}
{\mu_{{2}}}^{4}}{\sigma_{{1}}{\sigma_{{22}}}^{3}+3\sigma_{{1}}{
\sigma_{{22}}}^{2}{\mu_{{2}}}^{2}+2\sigma_{{1}}\sigma_{{22}}{\mu
_{{2}}}^{4}+{\sigma_{{22}}}^{3}{\mu_{{1}}}^{2}-2{\sigma_{{22}}}^{2
}\mu_{{1}}\sigma_{{2}}\mu_{{2}}-{\sigma_{{22}}}^{2}{\sigma_{{2}}}
^{2}-2\sigma_{{22}}{\sigma_{{2}}}^{2}{\mu_{{2}}}^{2}-2{\sigma_{
{2}}}^{2}{\mu_{{2}}}^{4}}}.
\end{align*}

The components $Ric_{ij}$, $i,j=1,2,3,4,5$ of the Ricci tensor are given by\\
$Ric_{11}=1/2\sigma_{{1}}{\alpha}^{2}-1/2\sigma_{{1}}$,\\
$Ric_{12}=1/2\sigma_{{2}}{\alpha}^{2}-1/2\sigma_{{2}}$,\\
$Ric_{13}=\mu_{{1}}\sigma_{{1}}{\alpha}^{2}-\mu_{{1}}\sigma_{{1}}$,\\
$Ric_{14}=\left( 1/2\sigma_{{1}}\mu_{{2}}+1/2\sigma_{{2}}\mu_{{1}}
 \right) {\alpha}^{2}-1/2\sigma_{{1}}\mu_{{2}}-1/2\sigma_{{2}}
\mu_{{1}}$,\\
$Ric_{15}=-\sigma_{{2}}\mu_{{2}}+{\alpha}^{2}\sigma_{{2}}\mu_{{2}}$,\\
$Ric_{22}=1/2\sigma_{{22}}{\alpha}^{2}-1/2\sigma_{{22}}$,\\
$Ric_{23}=-\sigma_{{2}}\mu_{{1}}+{\alpha}^{2}\sigma_{{2}}\mu_{{1}}$,\\
$Ric_{24}=\left( 1/2\sigma_{{22}}\mu_{{1}}+1/2\sigma_{{2}}\mu_{{2}}
 \right) {\alpha}^{2}-1/2\sigma_{{22}}\mu_{{1}}-1/2\sigma_{{2}}
\mu_{{2}}$,\\
$Ric_{25}=-\sigma_{{22}}\mu_{{2}}+\sigma_{{22}}{\alpha}^{2}\mu_{{2}}$,\\
$Ric_{33}=2\sigma_{{1}} \left( \sigma_{{1}}+{\mu_{{1}}}^{2} \right) {
\alpha}^{2}-2\sigma_{{1}} \left( \sigma_{{1}}+{\mu_{{1}}}^{2}
 \right)$,\\
$Ric_{34}=\left( 2\sigma_{{2}}\sigma_{{1}}+\mu_{{1}}\sigma_{{1}}\mu_{{2
}}+{\mu_{{1}}}^{2}\sigma_{{2}} \right) {\alpha}^{2}-2\sigma_{{2}
}\sigma_{{1}}-\mu_{{1}}\sigma_{{1}}\mu_{{2}}-{\mu_{{1}}}^{2}\sigma
_{{2}}$,\\
$Ric_{35}=\left( -\sigma_{{1}}\sigma_{{22}}+3{\sigma_{{2}}}^{2}+2\mu_{
{1}}\sigma_{{2}}\mu_{{2}} \right) {\alpha}^{2}+\sigma_{{1}}\sigma_
{{22}}-3{\sigma_{{2}}}^{2}-2\mu_{{1}}\sigma_{{2}}\mu_{{2}}$,
\begin{align*}
Ric_{44}&=\left( 3/2\sigma_{{1}}\sigma_{{22}}+1/2\sigma_{{1}}{\mu_{{2}
}}^{2}+1/2{\mu_{{1}}}^{2}\sigma_{{22}}+1/2{\sigma_{{2}}}^{2}+
\mu_{{1}}\sigma_{{2}}\mu_{{2}} \right) {\alpha}^{2}-3/2\sigma_{{1
1}}\sigma_{{22}}-1/2\sigma_{{1}}{\mu_{{2}}}^{2}-1/2{\mu_{{1}}}^
{2}\sigma_{{22}}\\
&\ \ \ -1/2{\sigma_{{2}}}^{2}-\mu_{{1}}\sigma_{{2}}\mu
_{{2}},
\end{align*}
$Ric_{45}=\left( 2\sigma_{{22}}\sigma_{{2}}+\sigma_{{22}}\mu_{{1}}\mu_{{2
}}+\sigma_{{2}}{\mu_{{2}}}^{2} \right) {\alpha}^{2}-2\sigma_{{22}
}\sigma_{{2}}-\sigma_{{22}}\mu_{{1}}\mu_{{2}}-\sigma_{{2}}{\mu_{{
2}}}^{2}$,\\
$Ric_{55}=2\sigma_{{22}} \left( \sigma_{{22}}+{\mu_{{2}}}^{2} \right) {
\alpha}^{2}-2\sigma_{{22}} \left( \sigma_{{22}}+{\mu_{{2}}}^{2}
 \right)$.\\
\begin{rem}
The computations of $N^3$, even $N^d$ for all $d\in\mathbb{N}$ are similar to that of $N^2$ (in fact, they are calculated by a common program), but the results are too large to be presented here.
\end{rem}

\end{document}